\documentclass[12pt]{amsart}
\usepackage{color}

\usepackage{amsmath,amsthm,amssymb,verbatim}
\usepackage{xcolor}
\usepackage[all]{xy}
\usepackage{tikz} \usepackage{tikz-cd}
\usepackage{graphicx}
\usetikzlibrary{matrix,arrows,shadows,backgrounds,positioning,decorations.markings,calc}
\usetikzlibrary{shapes}
\usepackage{hyperref}

\def\TC{\protect\operatorname{TC}}
\def\secat{\protect\operatorname{secat}}
\def\Imm{\protect\operatorname{Imm}}
\DeclareMathOperator{\cat}{\mathrm{cat}}
\DeclareMathOperator{\zcl}{\mathrm{zcl}}
\DeclareMathOperator{\wgt}{\mathrm{wgt}}
\DeclareMathOperator{\zclw}{\zcl^{\underline{w}}_s}

\def\Z{{\mathbb{Z}}}

\DeclareRobustCommand{\arraba}{\genfrac{}{}{0pt}{}}

\newcommand{\fv}{\mathfrak v}
\newcommand{\fb}{\mathfrak b}

\newtheorem{theorem}{Theorem}[section]
\newtheorem{lemma}[theorem]{Lemma}

\newtheorem{proposition}[theorem]{Proposition}
\newtheorem{definition}[theorem]{Definition}
\newtheorem{corollary}[theorem]{Corollary}

\numberwithin{equation}{section}

\newtheorem{theo}{Theorem}

\theoremstyle{definition}
\newtheorem{example}[theorem]{Example}
\newtheorem{remark}[theorem]{Remark}

\begin{document}
\title[Higher TC of manifolds  with abelian fundamental group]{On the higher topological complexity of manifolds  with abelian fundamental group}

\author[N. Cadavid-Aguilar]{N. Cadavid-Aguilar}
\address{Departamento de Matem\'aticas, Centro de Investigaci\'on y de Estudios Avanzados del IPN, Av. Instituto Polit\'ecnico Nacional 2508, San Pedro Zacatenco, Ciudad de M\'exico 07000}
\email{ncadavia@gmail.com}
\urladdr{}

\author[D. Cohen]{D. Cohen}
\address{Department of Mathematics, Louisiana State University, Baton Rouge, Louisiana 70803}
\email{\href{mailto:cohen@math.lsu.edu}{cohen@math.lsu.edu}}
\urladdr{\href{http://www.math.lsu.edu/~cohen/}
{www.math.lsu.edu/\char'176cohen}}

\author[J.  Gonz\'alez]{J. Gonz\'alez}
\address{Departamento de Matem\'aticas, Centro de Investigaci\'on y de Estudios Avanzados del IPN, Av. Instituto Polit\'ecnico Nacional 2508, San Pedro Zacatenco, Ciudad de M\'exico 07000}
\email{jesus.glz-espino@cinvestav.mx}
\urladdr{}

\author[S.  Hughes]{S. Hughes}
\address{Rheinische Friedrich-Wilhelms-Universit\"at Bonn, Mathematical Institute, Endenicher Allee 60, 53115 Bonn, Germany}
\email{sam.hughes.maths@gmail.com; \href{mailto:hughes@math.uni-bonn.de}{hughes@math.uni-bonn.de}}
\urladdr{\href{https://samhughesmaths.github.io}{samhughesmaths.github.io}}

\author[L. Vandembroucq]{L. Vandembroucq}
\address{Centro de Matem\'atica, Universidade do Minho, Braga, Portugal}
\email{lucile@math.uminho.pt}
\urladdr{}

\begin{abstract} 
We study the higher (or sequential) topological complexity $\TC_s$ of manifolds with abelian fundamental group. We give sufficient conditions for $\TC_s$ to be non-maximal in both the orientable and non-orientable cases. In combination with cohomological lower bounds, we also obtain some exact values for certain families of manifolds.  
\end{abstract}

\maketitle

\section*{Introduction}

For a path-connected space $X$, the $s$-th \emph{higher topological complexity} $\TC_s(X)$ is the sectional category of the fibration $e_s\colon PX\to X^s$, that is
\[
\TC_s(X)=\secat(e_s\colon PX\to X^s), 
\]
where $PX$ denotes the space of paths in $X$ and
\[
e_s(\gamma)=\Bigl(\gamma(0),\gamma\bigl(\frac{1}{s-1}\bigr),\ldots,\gamma\bigl(\frac{s-2}{s-1}\bigr),\gamma(1)\Bigr)
\]
is the usual $s$-th evaluation map. That is, in the reduced version used here,  $\TC_s(X)$ is one less than the minimal number of open sets covering $X^s$, over each of which the fibration $e_s$ admits a section.

Topological complexity $\TC(X)=\TC_2(X)$ was introduced by Farber in \cite{Farber} and the `higher' invariants were introduced by Rudyak in \cite{Rudyak}.  The invariants were developed and motivated by applications for motion planning problems in robotics.   More precisely, viewing $X$ as the space of configurations of a mechanical system, the integer $\TC_s(X)$ provides a topological measure of the complexity of planning motion in $X$ from an initial configuration to a terminal configuration, passing through $s-2$ specified intermediate configurations.

Despite a huge body of research into these invariants, there are very few complete computations of $\TC_s(X)$. Examples for which the full spectrum of invariants is known include products of spheres, surfaces, path-connected topological groups whose Lusternik-Schnirelmann category is known, closed simply-connected symplectic manifolds, classifying spaces of hyperbolic groups and some (additional) polyhedral product type spaces, see \cite{BGRT, GGGL, HL, AGO}. In a number of these examples, the higher topological complexities attain the maximal values possible. 

If $X$ is not simply connected, this maximal value is $\TC_s(X) \le s \dim(X)$, where $\dim(X)$ is the homotopy dimension of $X$, see \cite{BGRT}. Work of Cohen--Vandembroucq \cite{CV21} explored the non-maximality of $\TC_2(M)$ when $M$ is a manifold with abelian fundamental group.  In this paper we extend these ideas to $\TC_s(M)$ for $s\ge 2$.

Espinosa Baro, Farber, Mescher, and Oprea \cite{EFMO} have recently characterized the maximality of $\TC_s$ of a finite-dimensional CW-complex $X$ in terms of a canonical cohomology class generalizing the `Costa--Farber class' introduced in \cite{MR2649230} (see Section \ref{secintro}). Restricting our attention to a manifold $M$ with an abelian fundamental group $\pi$ and following the strategy of \cite{CV21}, we first  
express 
this characterization in terms of a homology class of the group $\pi^{s-1}$ (see Proposition \ref{suff-cond} and Corollary \ref{Cor-orientable case}). This permits us to establish the non-maximality of $\TC_s(M)$ in some cases. For example, when $M$ is orientable, we obtain the following result (see Section \ref{sec:orient}):
\begin{theo}\label{intro_upperbound_orientable} Let $M$ be an orientable $n$-dimensional connected closed manifold. In each of the following cases, we have $\TC_s(M)<sn$:
	\begin{enumerate}
		\item $\pi_1(M)=\Z^r$ with $(s-1)r<s\dim(M)$;
		\item $\pi_1(M)=\Z_q$;
		\item $\pi_1(M)=\Z^ r\times \Z_q$ with $r<\dim(M)$.
	\end{enumerate} 
\end{theo}

Computations of cohomological lower bounds of the $s$-th topological complexity of the real projective spaces and lens spaces have attracted much interest \cite{CGGGL, Davis, FG, Daundkar}.  In Section \ref{sec:genmfld}, we show how these results provide lower bounds of $\TC_s$ for larger families of manifolds (see Proposition \ref{prop:generalmanifold} and Theorem~\ref{zclw_M_Lens}). Then Theorem \ref{intro_upperbound_orientable} enables us to obtain the following exact values:
\begin{theo} Let $M$ be an orientable $n$-dimensional connected closed manifold with maximal Lusternik--Schnirelmann category, that is, $\cat(M)=n$.
\begin{enumerate}
\item If $n\equiv 1$ mod $4$ and $\pi_1(M)=\Z_2$, then $\TC_s(M)=sn-1$ for $s$ sufficiently large.
\item If $n\equiv 1$ mod $2$ and $\pi_1(M)=\Z_p$ where $p\geq 3$ is a prime, then $\TC_s(M)=sn-1$ for $s$ sufficiently large.
\end{enumerate}
\end{theo}
See Corollaries \ref{Exorientable} and \ref{exactvalueZm} for a more explicit description of the condition ``$s$ sufficiently large''.

The case of non-orientable manifolds is much more complicated. By \cite[Theorem~1.2(1)]{CV21}, the topological complexity of a non-orientable manifold with abelian fundamental group is always non-maximal.  However, it is well-known that there exist such non-orientable manifolds with maximal $\TC_s$ for $s\geq 3$. For instance, for the real projective plane $P^2$, we have $\TC_3(P^2)=6=3\dim (P^2)$, see \cite{GGGL}. 
Furthermore, 
it has been shown in \cite{CGGGL} and \cite{Davis} that, when $n$ is even, the real projective space $P^n$ satisfies $\TC_s(P^n)=sn$ for $s$ sufficiently large. For a fixed even integer $n$, the sequence $(\TC_s(P^n))_{s\geq 2}$ forms an increasing sequence starting at $\TC_2(P^n)$, equal to the immersion dimension of $P^n$ (\cite{FTY}), and stabilizing to $sn$ when $s$ is sufficiently large. As explained 
in \cite{CGGGL}, it would be very interesting to better understand this sequence. Our methods, developed in Section \ref{no-orient}, permit us to obtain new information in this direction. In particular, in combination with Davis' results \cite{Davis}, when $n=2^r-2$, we have:
\[\TC_s(P^n)\leq sn-1 \ \mbox{for even } s\leq n, \quad \TC_n(P^n)=n^2-1, \quad \mbox{and} \quad \TC_s(P^n)=sn \ \mbox{for } s>n.\]
As before, this result can be extended to a larger family of manifolds:
\begin{theo} Let $M$ be a non-orientable $n$-dimensional connected closed manifold with $\pi_1(M)=\mathbb{Z}_2$ and $n=2^{r}-2$ where $r\geq 3$. Then, for any even $s \leq n$, we have $\TC_s(M)<sn$.
	If moreover  $\cat{M}=n$, then $\TC_n(M)=n^2-1$ and $\TC_s(M)=sn$ for $s>n$. 
\end{theo}

\subsection*{Notation and conventions} For a topological space $Y$, 
we use $\dim(Y)$ to denote the homotopy dimension of $Y$. The integral homology of $Y$ is denoted by $H_*(Y)$, and the reduced homology by $\widetilde{H}_*(Y)$. If $\pi=\pi_1(Y)$, we denote the cohomology of $Y$ with coefficients in the local system 
determined by the $\Z[\pi]$-module $V$ by $H^*(Y;V)$.

For an element $a$ of a group $\pi$, we often denote the inverse of $a$ by $\overline{a}$.

We use the reduced version of sectional category throughout, so that for a fibration $p\colon E \to B$, when finite, $\secat(p\colon E \to B)$ is one less than the minimal number of open sets covering $B$, over each of which the fibration admits a section.

\section{A \texorpdfstring{$\TC_s$}{higher TC} canonical class}\label{secintro}

Canonical cohomology classes for higher topological complexity were recently introduced and studied by 
Espinosa Baro, Farber, Mescher, and Oprea, see \cite{EFMO}. In this brief preliminary section, with this work as a general reference (\cite[\S\S 5--6]{EFMO} in particular), we recall and discuss aspects of these classes which will be of subsequent use.

Let $X$ be a CW-complex.
The standard dimensional upper bound for higher topological complexity is 
\begin{equation}\label{estimacioninicial}
\TC_s(X)\leq s\dim(X).
\end{equation}

Although~(\ref{estimacioninicial}) can be improved in terms of the connectivity of $X$, we are interested in the improvements coming from obstruction-theory techniques in cases where $X$ is not simply connected.  {A fundamental concept in this context is the notion of homological obstruction as considered in Schwarz' monograph~\cite{Sch58}.
Recall that the fiber of $e_s\colon PX\to X^s$ is $\Omega X^{s-1}=(\Omega X)^{s-1}$.
 In \cite{EFMO}, the homological obstruction for sectioning $e_s$ over the 1-dimensional skeleton of $X^s$ is identified with a canonical twisted class,
\begin{equation}\label{ppp}
\fv_{X,s}\in H^1(X^s;I_s(\pi^{s-1}))= H^1(X^s;\widetilde{H}_0(\Omega X^{s-1})),
\end{equation}
where $\pi:=\pi_1(X)$ and $I_s(\pi^{s-1})$ denotes the augmentation ideal of $\pi^{s-1}$, viewed as a $\Z[\pi^s]$-submodule of $\Z[\pi^{s-1}]$. Here the action of $\pi^s$ on $I_s(\pi^{s-1})$, which corresponds to the 
monodromy associated with the fibration $e_s$, is given by 
\[
(a_1,\ldots,a_s)\cdot(b_1,\ldots,b_{s-1})=(a_1b_1\overline{a_2},a_2b_2\overline{a_3},\ldots,a_{s-1}b_{s-1}\overline{a_s}).
\]

The class $\fv_{X,s}$ can also be described as the cohomology class induced by the crossed homomorphism
 $\nu_{X,s}\colon\pi^s\to I_s(\pi^{s-1})$ given by $$\nu_{X,s}(a_1,\ldots,a_s)=\left(a_1\overline{a_2},a_2\overline{a_3},\ldots,a_{s-1}\overline{a_s}\right)-1_{s-1},$$
where $1_{s-1}$ is the unit element of $\pi^{s-1}$. Obstruction-theoretic arguments lead then to the following result: 
\begin{theorem}[\cite{EFMO}] \label{EFMO} Let $X$ be a CW-complex of dimension $n\geq 2$. Then $\TC_s(X)< sn$ if and only if the $sn$-th cup-power $\fv_{X,s}^{sn}=0$.
\end{theorem}
Here, 
$\fv_{X,s}^{sn}$ lies in the cohomology of $X^s$ with coefficients in 
the $sn$-th tensor power of $I_s(\pi^{s-1})$ endowed with the diagonal action of $\pi_s$, denoted by $I_s^{sn}(\pi^{s-1})$.}

The construction of the class $\fv_{X,s}$ generalizes the $\TC$ canonical class of \cite{MR2649230} and provides for $\TC_s$ an analogue of the classical Berstein-Schwarz class $\fb_X\in H^1(X;I(\pi))$. Note that, in this case, $I(\pi)$ is the augmentation ideal of $\pi$ endowed with the left $\Z[\pi]$-module structure induced by the multiplication of $\pi$. As is well-known, the Lusternik--Schnirelmann category of $X$, $\cat(X)$, satisfies $\cat(X)=\dim X$ if and only if  $\fb_X^{\dim(X)}\neq 0$ (see \cite{Berstein}, \cite{Sch58} and \cite{DrRud} for a proof including the case $\dim(X)= 2$).

\begin{remark} \label{rem:crossed}
We conclude this section with a brief remark regarding the functoriality of these classes. For $\pi=\pi_1(X)$, applying the above result to the classifying space $B\pi$ yields a crossed homomorphism and associated cohomology class, which we denote by $\nu_{\pi,s}$ and  $\fv_{\pi,s}$ respectively.

Recall that if $f:X\to Y$ is a map, and $A$ is a $\Z[\pi_1(Y)]$-module, then $f^*(A)$ denotes the $\Z[\pi_1(X)]$-module whose underlying abelian group is $A$ and the action of $g\in \pi_1(X)$ on $a\in A$ is given by $g\cdot a:=\pi_1(f)(g)\cdot a$. 
Taking $f\colon X \to Y = B\pi$ to be a classifying map, the isomorphism ${I}_s(\pi^{s-1})\cong (f^s)^*{I}_s(\pi^{s-1})$ yields
\[
{\mathfrak v}_{X,s}=(f^s)^*{\mathfrak v}_{\pi,s}.
\]
Similar considerations apply to the Berstein-Schwarz class $\fb_\pi\in H^1(\pi;I(\pi))$ (resp., $\fb_X\in H^1(X;I(\pi))$), induced by the crossed homomorphism $\beta_\pi\colon \pi \to I(\pi)$, $\alpha \mapsto \alpha-1$. Namely, the isomorphism $I(\pi)\cong f^*I(\pi)$ yields $\fb_X=f^*\fb_\pi$. 
\end{remark}

\section{Abelian fundamental group}

In this section we extend to higher topological complexity some results of \cite{CV21} which will be useful for our computations. The arguments are therefore similar to those of \cite{CV21} as well as some of \cite{Dranishnikov}.

Assume from now on that $\pi=\pi_1(X)$ is abelian. 
We consider the group homomorphism $^s\chi: \pi^s \to \pi^{s-1}$ given by
\[^s\chi(a_1,\dots,a_s)=(a_1\overline{a_2}, a_2\overline{a_3},\dots , a_{s-1}\overline{a_s}).\]
Note that the $\Z[\pi^s]$-module $^s\chi^*(I(\pi^{s-1}))$ is exactly the $\Z[\pi^s]$-module ${I}_s(\pi^{s-1})$. With the notation regarding canonical classes, Berstein-Schwarz classes, and crossed homomorphisms of the previous section,  we also have, for any $(a_1,\dots,a_s)\in \pi^s$,
\[\nu_{\pi,s}(a_1,\dots,a_s)=\beta_{\pi^{s-1}}(a_1\overline{a_2}, a_2\overline{a_3},\dots , a_{s-1}\overline{a_s})=\beta_{\pi^{s-1}}({^s\chi}(a_1,\dots,a_s)).\]
We then have ${\mathfrak v}_{\pi,s}={{}^s\chi}^*{\mathfrak b}_{\pi^{s-1}}$ in $H^1(\pi^s;{I}_s(\pi^{s-1}))=H^1(B\pi^s;{I}_s(\pi^{s-1}))$, and, for any $k$,
\[{\mathfrak v}^k_{X,s}= (\gamma^s)^*{\mathfrak v}^k_{\pi,s}= (\gamma^s)^*{(^s\chi)}^*{\mathfrak b}^k_{\pi^{s-1}} \,\, \mbox{ in } H^k(X^s;{I}_s^k(\pi^{s-1}))\]
where $\gamma:X\to B\pi$ is a classifying map.

In order to establish our results, it is useful to consider the cofiber of the diagonal map {$\Delta_s=\Delta_s^X:X\to X^s$}. We denote it by $C_{\Delta_s}(X)$.
{We will more generally use the notation $\Delta_s^Z:Z\to Z^s$ to denote the $s$-diagonal of a set $Z$ and suppress the superscript when the context is clear.}    

\begin{proposition}\label{nec-cond}
Let $X$ be an $n$-dimensional CW-complex with $n\geq 2$. Suppose that $\pi=\pi_1(X)$ is abelian and let $\gamma:X\to B\pi$ be a classifying map. Then for any $s\geq 2$ we have
\begin{enumerate}
	\item $\mathfrak{v}_{X,s}=q^*\mathfrak{b}_{C_{\Delta_s}(X)}$ in $H^1(X^s; I_s(\pi^{s-1}))$ where $q: X^s\to C_{\Delta_s}(X)$ is the identification map.
	\item $\TC_s(X)<sn$ if and only if $\cat(C_{\Delta_s}(X))<sn$.
	\item If $\TC_s(X)< sn$ then, for any $\Z[\pi^{s-1}]$-module $A$ and for any homology class $\mathsf{c}\in H_{sn}(X^s;(^s\chi\gamma^s)^*A)$, the class $\mathfrak{c}=\gamma^s_*(\mathsf{c})\in H_{sn}(\pi^s;(^s\chi)^*A)$ satisfies $^s\chi_*(\mathfrak{c})=0$.
\end{enumerate}
\end{proposition}

\begin{proof}

First observe that the homomorphism $^s\chi {\circ\Delta_s^{\pi}}$ is trivial. Consequently, the map $B^s\chi \circ B{{\Delta_s^{\pi}}}$ obtained after applying the functor $B$ is also trivial. By identifying $B\pi^s$ with $(B\pi)^s$ and $B{\Delta_s^{\pi}}$ with $\Delta_s^{B\pi}$, we have a commutative diagram of the following form	
	\[\xymatrix{
			X \ar[r]^{\Delta_s^X} \ar[d]_{\gamma} &	X^s \ar[d]_{\gamma^s} \ar[r]^{q} & C_{\Delta_s}(X) \ar[d]_{\xi}\\
			B\pi  \ar[r]^{\Delta_s^{B\pi}} & (B\pi)^s \ar[r]^{{B^s\chi}} & B\pi^{s-1}
		}\]
	where $\xi$ is induced by the quotient property.
	Since $\pi$ is abelian, we have an exact sequence $1\to \pi \stackrel{\Delta_s^{\pi}}{\longrightarrow} {\pi^s} \stackrel{^s\chi}{\longrightarrow} {\pi^{s-1}} \to 1$ and, using the Van Kampen theorem, we can see that $\pi_1(C_{\Delta_s}(X))=\pi^{s-1}$ and that $\pi_1(\xi)$ is an isomorphism. Consequently $\xi$ is a classifying map and the Berstein-Schwarz class of $C_{\Delta_s}(X)$ is given by
			${\mathfrak b}_{C_{\Delta}}=\xi^*{\mathfrak b}_ {\pi^{s-1}}$. By the commutativity of the diagram we then get $q^*(I(\pi^{s-1}))\cong {^s\chi}^*(I(\pi^{s-1}))\cong I_s(\pi^{s-1})
			$  and $q^*{\mathfrak b}_{C_{\Delta_s}(X)}=\mathfrak{v}_{X,s}$ as claimed in the first item. 
   
The equality established above implies that  $q^*{\mathfrak b}^{sn}_{C_{\Delta_s}(X)}=\mathfrak{v}^{sn}_{X,s}$. 
   For dimensional reasons, the map $q^*: H^{sn}(C_{\Delta_s}(X);I(\pi^{s-1}))\to H^{sn}(X^s;I_s(\pi^{s-1}))$ is an isomorphism. We therefore have ${\mathfrak b}^{sn}_{C_{\Delta_s}(X)}=0$ if and only if $\mathfrak{v}^{sn}_{X,s}=0$, which implies the second item. 
   
   We now prove the last item. Let $\mathsf{c}\in H_{sn}(X^s;(^s\chi\gamma^s)^*A)$ be a nonzero class and let $\mathfrak{c}=\gamma^s_*(\mathsf{c})$. We have $^s\chi_*(\mathfrak{c})=\xi_*q_*(\mathsf{c})$. Note that $q_*(\mathsf{c})$ is a homology class of degree $sn$. Since $\TC_s(X)< sn$ we have $\cat(C_{\Delta_s}(X))<sn$. Therefore the classifying map $\xi$ factors up to homotopy through an $(sn-1)$-dimensional space. Consequently, $\xi_*q_*(\mathsf{c})=0$ and the result follows. \end{proof}

\begin{remark} {In the situation of Proposition \ref{nec-cond}, if $A$ is a trivial $\Z[\pi^{s-1}]$-module and $\mathsf{c}\in H_{sn}(X^s;A)$ is an element such that the class $\mathfrak{c}=\gamma^s_*(\mathsf{c})\in H_{sn}(\pi^s;A)$ satisfies $^s\chi_*(\mathfrak{c})\neq 0$ then $\TC_s(X)=sn$.}
\end{remark}
			
Item (3) of Proposition \ref{nec-cond} is sharp under reasonably general conditions. Let $M$ be an $n$-dimensional connected closed manifold with fundamental group $\pi=\pi_1(M)$ and let $\omega=\omega_M: \pi\to \{\pm 1\}$ be the homomorphism determined by the first Stiefel-Whitney class of $M$. Recall that the orientation module of $M$, denoted by $\widetilde{\Z}=\widetilde{\Z}_M$, is the abelian group $\Z$ given with a structure of $\Z[\pi]$-module determined by $a\cdot t=\omega(a)t$ for $a\in \pi$, $t\in \Z$. Note that $\widetilde{\Z}_{M^s}=\widetilde{\Z}_{M}^{\otimes s}$, which additively is $\Z$ with $\pi^s$ action given by $(a_1,a_2,\dots,a_s)t=\omega(a_1)\omega(a_2)\cdots \omega(a_s)t$. 

\begin{proposition} \label{suff-cond} Let $M$ be an $n$-dimensional connected closed manifold with $n\geq 2$ and $\pi=\pi_1(M)$ abelian. Assume there is a $\Z[\pi^{s-1}]$-module $A$ such that {the $\Z[\pi^s]$-modules $^s\chi^*(A)$ and $\widetilde{\Z}^{\otimes s}$ are isomorphic.} Then the following two conditions are equivalent:
\begin{enumerate}
\item The class $\mathfrak{m}:=\gamma_*([M])\in H_n(\pi;\widetilde{\Z})$ satisfies $^s\chi_*(\mathfrak{m}^{\times s}) =0$ in $H_{sn}(\pi^{s-1};A)$.
\item $\TC_s(X)<sn$. 
\end{enumerate}
\end{proposition}

\noindent Here we denote by $\mathfrak{m}^{\times s}\in H_{sn}(\pi^s;\widetilde{\Z}^{\otimes s})$ the image of the fundamental class of $M^s$ under the homomorphism induced by $\gamma^s:M^s\to B\pi^s$. Note that we also denote by $\widetilde{\Z}$ the local system over $B\pi$ arising from the isomorphism $\pi_1(\gamma)$ induced by the classifying map $\gamma\colon M \to B\pi$.  

\begin{proof}
From the naturality of the cap-product and the assumption that $A$ is a $\Z[\pi^{s-1}]$-module satisfying $^s\chi^*(A)\cong\widetilde{\Z}^{\otimes s}$	we get the following diagram.

\begin{equation*}
\xymatrixcolsep{5pc}
\xymatrix{
	H_{sn}(M^s;\widetilde{\Z}^{\otimes s}) \otimes H^{sn}(M^s;{I}_s^{sn}(\pi^{s-1})) \ar@<-8ex>[d]^{(\gamma^s)_*} \ar[r]^-{\cap}_-{\cong} & 
	{I}_s^{sn}(\pi^{s-1})\otimes_{\pi^s}\widetilde{\Z}^{\otimes s} \ar[d]_{=} \\
	H_{sn}(B\pi^s;\widetilde{\Z}^{\otimes s}) \otimes H^{sn}(B\pi^s;{I}_s^{sn}(\pi^{s-1}) ) \ar@<-8ex>[u]_{(\gamma^s)^*} 
	\ar@<-8ex>[d]^{(^s\chi)_*}
	\ar[r]^-{\cap} & 
	{I}_s^{sn}(\pi^{s-1})\otimes_{\pi^s}\widetilde{\Z}^{\otimes s}
	\ar[d]^{(^s\chi)_*}_{\cong}\\
		H_{sn}(B\pi^{s-1};A) \otimes H^{sn}(B\pi^{s-1};{I}^{sn}(\pi^{s-1}) ) \ar@<-8ex>[u]_{(^s\chi)^*} 
	\ar[r]^-{\cap} & 
	{I}^{sn}(\pi^{s-1})\otimes_{\pi^{s-1}}A\\
}
\end{equation*}
The cap-product on the first line is an isomorphism by Poincar\'e duality.
The bottom vertical map in the third column corresponds to the morphism $$\chi_*: H_*(B\pi^s;{I}_s^{sn}(\pi^{s-1})\otimes\widetilde{\Z}^{\otimes s} )\to H_*(B\pi^{s-1}; {I}^{sn}(\pi^{s-1})\otimes A)$$ in degree $0$. It is induced by the obvious isomophism between the underlying $\Z$-modules $I^{sn}(\pi^{s-1})\otimes \widetilde{\Z}^{\otimes s}$ and $I^{sn}(\pi^{s-1})\otimes A$ and is an isomorphism on the coinvariants because {$^s\chi$} is surjective.

Let $[M]\in H_n(M; \widetilde{\Z})$ be the fundamental class. Since the third column of the diagram is comprised of isomorphisms, we have
\[(^s\chi)_*(\gamma^s)_*([M^s])\cap \mathfrak{b}^{sn}_{\pi^{s-1}}=0  \mbox{ if and only if }[M^s]\cap (\gamma^s)^*(^s\chi)^*\mathfrak{b}^{sn}_{\pi^{s-1}}=0.
\]
This is equivalent to saying
\[(^s\chi)_*(\mathfrak{m}^{\times s})\cap \mathfrak{b}^{sn}_{\pi^{s-1}}=0  \mbox{ if and only if }[M^s]\cap \mathfrak{v}^{sn}_{M,s}=0.
\]
The hypothesis $^s\chi_*(\mathfrak{m}^{\times s}) =0$ yields $[M^s]\cap \mathfrak{v}^{sn}_{M,s}=0$. By Poincar\'e duality, we can then conclude that $\mathfrak{v}^{sn}_{M,s}=0$ and consequently $\TC_s(M)<sn$.
\end{proof}

\begin{corollary} \label{Cor-orientable case} Let $M$ be an orientable $n$-dimensional manifold with $n\geq 2$ and abelian fundamental group $\pi=\pi_1(M)$.  The class $\mathfrak{m}=\gamma_*([M])\in H_n(\pi;\Z)$ satisfies $^s\chi_*(\mathfrak{m}^{\times s}) =0$ in $H_{sn}(\pi^{s-1};\Z)$ if and only if $\TC_s(X)<sn$. 
	
\end{corollary}

\begin{proof}
	Since $M$ is orientable, the orientation module $\widetilde{\Z}$ is just $\Z$ with trivial action. Taking $A=\Z$ also with trivial action, we have 
$^s\chi^*(A)\cong{\Z}^{\otimes s}$ and the result follows from Proposition \ref{suff-cond}.
\end{proof}

\begin{remark}\label{elsistemadecorficientes}{ When $M$ is non-orientable, {the $\Z[\pi^{s-1}]$-module} $A=\widetilde{\Z}^{\otimes s-1}$ {satisfies the assumptions of Proposition \ref{suff-cond} for $s=2$ (\cite{CV21}) but fails to do so for all $s>2$}. 
		 For instance, set $s=3$ and suppose that $b\in \pi$ is an element for which the orientation character $\omega:\pi \to \{\pm1\}$ satisfies $\omega(b)=-1$. 
		Then, for $a,c\in \pi$ and $t\in \Z$ we have
		\[(a,b,c)\cdot t=\omega(a)\omega(b)\omega(c)t=-\omega(a)\omega(c)t\] 
		while
		\[^3\chi(a,b,c)\cdot t= (a\overline{b},b\overline{c})\cdot t= \omega(a\overline{b})\omega(b\overline{c})t=\omega(a)\omega(c)t.\] 
This shows that the map $^3\chi$ does not induce a homomorphism from $H_*(\pi^3;\widetilde{\Z}^{\otimes 3})$ to $H_*(\pi^2;\widetilde{\Z}^{\otimes 2})$.
{Note that, in Proposition \ref{suff-cond}, $A$ must be, {as an abelian group,}  isomorphic to $\Z$. Furthermore, since $^s\chi$ is surjective, the} $\Z[\pi^{s-1}]$-module structure {on $A$ is forced by the hypothesis} $^s\chi^*(A)\cong\widetilde{\Z}^{\otimes s}$ and this condition is impossible when $s$ is odd. 
For instance, again set $s=3$, choose $b\in\pi$ as above and assume $^3\chi^*(A)\cong\widetilde{\Z}^{\otimes 3}$. The equalities $^3\chi(b,1,1)={}^3\chi(1,\overline{b},\overline{b})=(b,1)$ then lead to the impossible
$$
t=\omega(\overline{b})\omega(\overline{b})t={}^3\chi(1,\overline{b},\overline{b})\cdot t=(b,1)\cdot t={}^3\chi(b,1,1)\cdot t=\omega(b)t=-t.
$$}
Nonetheless, when $s=2\sigma$, $\sigma\geq1$, the $\Z[\pi^{s-1}]$-module $A=\widetilde{\Z}\otimes(\Z\otimes\widetilde{\Z})^{\sigma-1}$ does satisfy $^s\chi^*(A)\cong\widetilde{\Z}^{\otimes s}$, and we explore its usage in Section \ref{no-orient}.
\end{remark}

\section{Some calculations in the orientable case} \label{sec:orient}

Let $M$ be an orientable connected closed manifold with $\pi=\pi_1(M)$ abelian. In this section we will use Corollary \ref{Cor-orientable case} to establish the non-maximality  $\TC_s(M)<s\dim(M)$ for some families of manifolds with abelian fundamental groups. 

Let $\gamma:M\to B\pi$ be a classifying map and let $\mathfrak{m}=\gamma_*([M])\in H_n(\pi;\Z)$. Since $M$ is orientable, we will suppress the $\Z$-coefficients from the notation. In all cases, we will see that $^s\chi_*(\mathfrak{m}^{\otimes s}) =0$ in $H_{sn}(\pi^{s-1})$.

In our first result, we suppose that $\pi$ is a free abelian group. This case has already been considered in \cite{EFMO} in the more general context of finite CW-complexes. Here, restricting to closed manifolds, we obtain a slightly stronger statement than \cite[Corollary 6.14]{EFMO}.

\begin{proposition}\label{free} Let $M$ be an orientable $n$-dimensional connected closed manifold with $\pi_1(M)=\Z^r$ and let $s\geq 2$. If $sn>(s-1)r$ then $\TC_{s}(M)<sn$.
\end{proposition}

\begin{proof} 
	Let $\pi=\Z^r$, let $\gamma:M\to B\pi$ be a classifying map and let $\mathfrak{m}=\gamma_*([M])\in H_n(\pi))$. For degree reasons, we can see that $^s\chi_*(\mathfrak{m}^{\times s}) =0$ in $H_{sn}(\pi^{s-1})$. Indeed $B\pi=(S^1)^r$ and $H_k(\pi)=0$ for $k>r$. Consequently  $H_{sn}(\pi^{s-1})=0$ if $sn>(s-1)r$.
\end{proof}

In general, observe that the homomorphism  $^s\chi: \pi^s \to \pi^{s-1}$ given by
\[^s\chi(a_1,\dots,a_s)=(a_1\overline{a_2}, a_2\overline{a_3},\dots , a_{s-1}\overline{a_s})\]
can be decomposed as 
\begin{equation}\label{decomposition-chis}
^s\chi=(\underbrace{\chi\times \cdots\times \chi}_{s-1})\circ({\rm Id}\times \underbrace{\Delta\times \cdots\times \Delta}_{s-2}\times {\rm Id})
\end{equation}
where $\Delta=\Delta_2^{\pi}:\pi \to \pi\times \pi$ is the diagonal map and $\chi={^2\chi}$. Denote by $j:\pi\to \pi$ the inversion.
Since $\chi$ can be seen as the multiplication of $\pi$, {$\mu:\pi \times \pi \to \pi$}, precomposed with ${\rm Id}\times j$, we have, for classes $ \mathfrak{a}, \mathfrak{b}\in H_{*}(\pi)$,  
\[\chi_*(\mathfrak{a}\times \mathfrak{b})=\mathfrak{a} \wedge j_*(\mathfrak{b})\]
where $\wedge$ is the Pontryagin product, {that is, the product induced by $\mu$ in homology}, see \cite[V.5]{Brown}. 	

In the results below, we consider the cyclic group $\Z_q=\langle v ~|~v^q=1\rangle$ and work at the chain level. Recall the classical resolution of $\Z$ as a trivial $\Z[\Z_q]$-module given by
\begin{equation}\label{resolutionZq}
\xymatrix{
	\cdots \ar[r] &\Z [\Z_q]\ar[r]^{N_q(v)} & \Z [\Z_q] \ar[r]^{v-1} & \Z [\Z_q] \ar[r]^{N_q(v)} & \Z [\Z_q] \ar[r]^{v-1}& \Z [\Z_q] \ar[r]^{\varepsilon}&\Z
}
\end{equation}
where $N_q(v)=1+v+\cdots+v^{q-1}$. 

In the following lemma, we recall the morphisms induced by the diagonal $\Delta$, the multiplication $\mu$ and the inversion $j$ on the level of resolutions (see \cite[page 108]{Brown} and  \cite[\S3.2]{CV21}). Let $[k]$ denote the generator of degree $k$ in \eqref{resolutionZq}, and write $B_{i,j}$ for the binomial coefficient $\binom{\,i+j\,}{i}$.

\begin{lemma}\label{deltaymuenresoluciones}
At the level of the resolution \eqref{resolutionZq}, 
\begin{enumerate}
    \item[$(a)$] 
$\Delta$ is given on generators by $[p] \mapsto \sum_{k+l=p}\Delta_{kl}[p]$ where
\[
\Delta_{kl}[p]=\left\{\begin{array}{ll}
[k]\otimes [l] & k \mbox{ even;}\\
\mathopen[ k\mathclose]\otimes v[l] & k \mbox{ odd, } l \mbox{ even;}\\
\sum_{0\leq i<j\leq q-1}v^i[k]\otimes v^j [l] & k \mbox{ odd, } l \mbox{ odd.}
\end{array}\right.
\]

 \item[$(b)$] $\mu$ is given on generators by the formul\ae
\begin{align}
[2i ]\otimes [2j]&\mapsto B_{i,j}\,[2(i+j)]; \nonumber \\
[2i]\otimes [2j+1]&\mapsto B_{i,j}\,[2(i+j)+1]; \nonumber \\
[2i+1]\otimes [2j]&\mapsto B_{i,j}\,[2(i+j)+1]; \nonumber \\
[2i+1]\otimes [2j+1]&\mapsto 0. \nonumber
\end{align}
 \item[$(c)$] $j$ is given on generators by 
$$
[i]\to N_{q-1}^k(v) [i] \qquad \mbox{if } i\in \{2k,2k-1\}.
$$
\end{enumerate}
\end{lemma}
We denote by $C_{\bullet}(\Z_q)$ the $\Z$-chain complex obtained by tensoring the resolution (\ref{resolutionZq}) with $\Z$ over $\Z_q$.
\begin{equation}\label{zcc}
\xymatrix{
	C_{\bullet}(\Z_q): &\cdots\ar[r]^-{0}&\Z[2k] \ar[r]^-{q}& \Z[2k-1] \ar[r]^-{0} &\cdots \ar[r]^{q}&\Z[1]  \ar[r]^{0}& \Z[0]  
}\end{equation}
Recall that the homology of this chain complex gives $H_*(\Z_q)=H_*(\Z_q;\Z)$. In positive degrees, $H_+(\Z_q)$ is concentrated in odd degrees.

As in \cite{CV21}, we denote by $\wedge: C_{\bullet}(\Z_q)\otimes C_{\bullet}(\Z_q)\to C_{\bullet}(\Z_q)$ the Pontryagin product, which is given by the formul\ae\ $(b)$ of Lemma \ref{deltaymuenresoluciones}:
\begin{equation}\label{productformula}
\left\{\begin{array}{l}
\arraycolsep=1.4pt\def\arraystretch{1.5}
~[2i]\wedge [2k]=B_{i,k}[2i+2k], \quad \quad [2i+1]\wedge [2k+1]=0, \\
\\
~[2i]\wedge [2k+1]=[2k+1]\wedge [2i]=B_{i,k}[2i+2k+1].
\end{array}\right.
\end{equation}
We denote by $\mathbf{j}:C_{\bullet}(\Z_q)\to C_{\bullet}(\Z_q)$ the morphism induced by the inversion, which is from Lemma \ref{deltaymuenresoluciones} $(c)$ given by
\begin{equation}\label{formulaj}
\mathbf{j}([i])= (q-1)^k [i] \qquad \mbox{if } i\in \{2k,2k-1\}.
\end{equation}
In these terms, the chain map $\chi_{\bullet}={^2\chi}_{\bullet}$ induced by $\chi={^2\chi}$ can be described as the composite
\[\xymatrix{
	C_{\bullet}(\Z_q)\otimes C_{\bullet}(\Z_q) \ar[r]^-{\mathrm{Id} \otimes \mathbf{j} }&
	C_{\bullet}(\Z_q)\otimes C_{\bullet}(\Z_q) \ar[r]^-{\wedge } &
	C_{\bullet}(\Z_q).
}\]
We will also use the diagonal approximation of $C_{\bullet}(\Z_q)$, obtained from Lemma \ref{deltaymuenresoluciones} $(a)$:
\begin{equation}\label{diagonalZq}
\begin{array}{rcl}
\Delta_{\bullet}: C_{\bullet}(\Z_q)&\to & C_{\bullet}(\Z_q)\otimes C_{\bullet}(\Z_q)\\[2pt]
{[p]} &\mapsto & \sum\limits_{k+l=p} \alpha_{kl}[k]\otimes [l].
\end{array} 
\end{equation}
Here $\alpha_{kl}=1$ if $kl$ is even and $\alpha_{kl}=(q-1)q/2$ if $kl$ is odd.\\

\begin{proposition}\label{cyclic} Let $M$ be an orientable $n$-dimensional connected closed manifold with $\pi_1(M)=\Z_q$. Then, for any $s\geq 2$, we have $\TC_{s}(M)<sn$.

\end{proposition}

\begin{proof}
Let $\gamma:M\to B\pi$ be a classifying map, where $\pi=\Z_q$, and let $\mathfrak{m}=\gamma_*([M])\in H_n(\pi;\Z))$. We will see that $^s\chi_*(\mathfrak{m}^{\times s}) =0$ in $H_{sn}(\pi^{s-1};\Z)$. If $\dim M$ is even, this is immediate since $H_+(\pi)$ is concentrated in odd degrees, which implies $\mathfrak{m}=0$. We then suppose that $\dim M=2p+1$. A cycle $\mathbf{m}\in C_{2p+1}(\Z_q)$ representing the class $\mathfrak{m}$ is of the form $\mathbf{m}=\lambda [2p+1]$ for some $\lambda \in \Z$. In order to compute $^s\chi_*(\mathfrak{m}^{\times s})$ we use the decomposition (\ref{decomposition-chis}) and analyze the element $^s\chi_{\bullet}(\mathbf{m}^{\otimes s})$ which is given by
\[
(\chi_{\bullet})^{\otimes s-1}(\mathbf{m}\otimes \underbrace{\Delta_{\bullet}\mathbf{m}\otimes \cdots\otimes \Delta_{\bullet}\mathbf{m}}_{s-2}\otimes \mathbf{m}).
\]
The element $\mathbf{m}\otimes \Delta_{\bullet}\mathbf{m}\otimes \cdots\otimes \Delta_{\bullet}\mathbf{m}\otimes \mathbf{m}$ is given by a $\Z$-linear combination of elements of the form
\[
[2p+1]\otimes [k_1]\otimes [l_1]\otimes \cdots \otimes [k_{s-2}]\otimes [l_{s-2}]\otimes [2p+1]
\]
where $k_i+l_i=2p+1$ for any $1\leq i\leq s-2$. Setting $l_0=k_{s-1}=2p+1$, there will be necessarily some $i\in\{0,\dots,s-2\}$ such that $l_i$ and $k_{i+1}$ are both odd. Applying $(\chi_{\bullet})^{\otimes s-1}$ to the element above yields
\[
([2p+1]\wedge \mathbf{j}[k_1])\otimes ([l_1]\wedge \mathbf{j}[k_2])\otimes \cdots \otimes ([l_{s-2}]\wedge \mathbf{j}[k_{s-1}]).
\]
If $l_i$ and $k_{i+1}$ are both odd, the corresponding factor $([l_i]\wedge \mathbf{j}[k_{i+1}])$ vanishes since $\mathbf{j}[k_{i+1}])$ is a multiple of $[k_{i+1}]$ and the Pontryagin product of two odd degree elements is zero. Consequently, we obtain $^s\chi_{\bullet}(\mathbf{m}^{\otimes s})=0$ and $^s\chi_*(\mathfrak{m}^{\times s})=0$.
\end{proof}

\begin{proposition}\label{freexcyclic} Let $M$ be an orientable $n$-dimensional connected closed manifold with $\pi_1(M)=\Z^r\times \Z_q$ such that $r<n$. Then, for any $s\geq 2$, we have $\TC_{s}(M)<sn$.
\end{proposition}

\begin{proof} Let $\gamma:M\to B\pi$ be a classifying map, where $\pi=\Z^r\times\Z_q$, and let $\mathfrak{m}=\gamma_*([M])\in H_n(\pi;\Z))$. By the K\"unneth formula, we have $H_*(\Z^r\times \Z_q)=H_*(\Z^r)\otimes H_*(\Z_q)$. Since $n>r$, we can write $\mathfrak{m}=\sum \sigma_i\otimes \alpha_i$ where  $\alpha_i\in H_{+}(\Z_q)$ and $\sigma_i\in H_*(\Z^ r)=\bigwedge(x_1,\dots,x_r)$ with each $x_j$ of degree $1$. Since $H_{+}(\Z_q)$ is 
	concentrated in odd degrees, a cycle $\mathbf{m}$ representing $\mathfrak{m}$ can be described as a sum of terms of the form $\lambda\sigma\otimes [2p+1]$ where $\lambda \in \Z$, $p\geq 0$ and $\sigma \in \bigwedge(x_1,\cdots,x_r)$ is a class regarded as a cycle.
The element $\mathbf{m}\otimes \Delta_{\bullet}\mathbf{m}\otimes \cdots\otimes \Delta_{\bullet}\mathbf{m}\otimes \mathbf{m}$ is therefore given by a $\Z$-linear combination of elements of the form
\begin{equation}\label{term_freexcyclic}
\sigma_0\otimes [l_0]\otimes \sigma_1\otimes [k_1]\otimes \tilde \sigma_1\otimes [l_1]\otimes \cdots \otimes \sigma_{s-2}\otimes\otimes [k_{s-2}]\otimes \tilde \sigma_{s-2}\otimes[l_{s-2}]\otimes \sigma_{s-1} [k_{s-1}]
\end{equation}
where $l_0$, $k_i+l_i$ for $1\leq i\leq s-2$, and $k_{s-1}$ are all odd and the elements $\sigma_i, \tilde\sigma_i$ belong to $\bigwedge(x_1,\cdots,x_r)$. The calculation of $\chi_{\bullet}$ on (say) $\sigma_0\otimes [l_0]\otimes \sigma_1\otimes [k_1]$ is made componentwise and gives rise to factors of the form
\[\pm(\sigma_0 \wedge \sigma_1)\otimes ([l_0]\wedge {\mathbf j}[k_1]).\]
	
As in the proof of Proposition \ref{cyclic}, there will be necessarily,  in the expression (\ref{term_freexcyclic}), some $i\in\{0,\dots,s-2\}$ such that $l_i$ and $k_{i+1}$ are both odd.  After applying $\chi_{\bullet}$,  the corresponding factor will be $0$. Consequently, we obtain $^s\chi_{\bullet}(\mathbf{m}^{\otimes s})=0$ and $^s\chi_*(\mathfrak{m}^{\times s})=0$. We can hence conclude that $\TC_s(M)<sn$.
\end{proof}

\subsection*{Limiting examples}
Examples 4.1  and 4.2 from \cite{CV21} show that the conditions in Propositions \ref{free} and \ref{freexcyclic} are sharp. We now show that Proposition \ref{cyclic} cannot be extended to manifolds whose fundamental group is of the form $\Z_p\times \Z_p$ where $p$ is a prime.

\begin{example} {\textit{ A manifold $N$ with $\pi_1(N)=\Z_3\times \Z_3$ and $\TC_3(N)=3\dim(N)$.}}{\rm
	
Set $\pi=\Z_3\times \Z_3$ and consider $C_{\bullet}(\pi)=C_{\bullet}(\Z_3)\otimes C_{\bullet}(\Z_3)$. We will write $[ik]$ instead of $[i]\otimes [k]$. We first consider the cycle $\mathbf{m}=[05]+[50]$ and denote by $\mathfrak{m}$ its homology class. We will see that $^3\chi_*(\mathfrak{m}^{\times 3})\neq 0$. By the Universal Coefficient Theorem, it is actually sufficient to see that $^3\chi_*(\mathfrak{m}^{\times 3}_{\Z_3})\neq 0$ where $\mathfrak{m}_{\Z_3}$ corresponds to $\mathfrak{m}$ in $H_5(\pi;\Z_3)$. As $H_*(\pi;\Z_3)\cong H_*(\Z_3;\Z_3)\otimes H_*(\Z_3;\Z_3)$ and $H_*(\Z_3;\Z_3)=\Z_3[k]$ for all $k\geq 0$, we will continue to write $\mathfrak{m}_{\Z_3}=[05]+[50]$. 

Using the diagonal approximation associated to the resolution (\ref{resolutionZq}) described in Lemma \ref{deltaymuenresoluciones} (or, tensoring the diagonal (\ref{diagonalZq}) by $\Z_q$) we can check that the homology diagonal of $H_*(\Z_3;\Z_3)$ satisfies
\[\Delta_*[0]=[0]\otimes [0] \qquad \Delta_*[5]=\sum\limits_{k+l=5}[k]\otimes [l]. \]
Consequently, the homology diagonal of $H_*(\pi;\Z_3)$ satisfies:
\[\Delta_*[05]=\sum\limits_{k+l=5}[0k]\otimes [0l]  \qquad \Delta_*[50]=\sum\limits_{k+l=5}[k0]\otimes [l0].\]
We have to compute: 
\[(\chi_*\otimes \chi_*)\left(([05]+[50])\otimes (\Delta_*[05]+\Delta_*[50])\otimes ([05]+[50])\right).\]
A term of the form $\chi_*([kl]\otimes [k'l'])$ is given in $H_*(\pi;\Z_3)=H_*(\Z_3;\Z_3)\otimes H_*(\Z_3;\Z_3)$ by a componentwise calculation:
\[\chi_*([kl]\otimes [k'l'])=(-1)^{lk'}([k]\wedge j_*[k'])\otimes ([l]\wedge j_*[l']).\]
Taking into account the formulas for the inversion and for the Pontryagin product (induced in $\Z_3$-homology by the formulas (\ref{formulaj}) and (\ref{productformula}) given above) we have
\[\chi_*([04]\otimes [05]])=([0]\wedge j_*[0])\otimes ([4]\wedge j_*[5])=[0] \otimes ([4]\wedge (-[5]))=-[0] \otimes (6[9])=-6[09]\]  
which vanishes since we are working with coefficients in $\Z_3$. We can thus check that 
\[(\chi_*\otimes \chi_*)\left([50]\otimes ([01]\otimes [04])\otimes ([05]+[50])\right)=[51]\otimes [54]\]
and that this is the only term belonging to $\Z_3[51]\otimes H_*(\pi;\Z_3)$ in the expansion of $^3\chi_*(\mathfrak{m}^{\times 3}_{\Z_3})$. Since $[51]\otimes [54]$ does not vanish in $\Z_3[51]\otimes H_*(\pi;\Z_3)$, we can conclude that 
$^3\chi_*(\mathfrak{m}^{\times 3}_{\Z_3})\neq 0$. Consequently $^3\chi_*(\mathfrak{m}^{\times 3})\neq 0$.

We can next follow the same strategy as in \cite{CV21} to show that there exists a manifold $N$ with fundamental group $\pi=\Z_3\times \Z_3$ and maximal $\TC_3$. More precisely, considering the lens spaces $L_3^5=S^5/\Z_3$ and $L_3^{\infty}=S^{\infty}/\Z_3=B\Z_3$, we can realize the class $[05]+[50]\in H_*(\pi)$ as the image of the fundamental class of $M=L_3^5\#L_3^5$ under the map induced by 
\[
f \colon M\xrightarrow{\ \text{pinch}\ } L_3^5\vee L_3^5 {\lhook\joinrel\longrightarrow} L_3^{\infty}\vee L_3^{\infty}=B(\pi).
\]
We can then use surgery to replace $M$ by a manifold $N$ with $\pi_1(N)=\pi$ and $f$ by a classifying map $\gamma: N\to B\pi$. In this way, $\mathfrak{m}=\gamma_*([N])$ and, from $^3\chi_*(\mathfrak{m}^{\times 3})\neq 0$ and Proposition \ref{nec-cond}\,(3), we can deduce that $\TC_3(N)=3n$.
}
\end{example}

In \cite{CV21}, it has been shown that the regular topological complexity $\TC=\TC_2$ of a non-orientable manifold with abelian fundamental group is never maximal. This is not longer true for $\TC_s$ with $s\geq 3$. For instance, for the real projective plane $P^2$, $s$-zero-divisor cuplength considerations imply that $\TC_3( P^2)=6$, see \cite{Davis} and the discussion in \S\ref{sec:genmfld} below. With the approach of this paper, we pursue more general maximality results of this nature next.

\section{Cohomological lower bounds}\label{sec:genmfld}
In this section, we use cohomological lower bounds on $\TC_s$ given by the $s$-zero-divisor cup length or $\TC_s$-weights as well as specific calculations from \cite{Davis, Daundkar} to obtain lower bounds on the higher topological complexity of families of manifolds with finite cyclic fundamental group and maximal LS-category. In some cases, exact values are given by using our results from Section \ref{sec:orient}.

Let $\Bbbk$ be a field. Recall that, for a space $X$, the ($\Bbbk$-coefficients) {$s$-zero-divisor cup length, $\zcl_s(X)=\zcl_s(X;\Bbbk)$, is} the maximum of the set
\[ \left\{ \ell\ |\ u_1\dots u_\ell \neq 0, u_i\in \mathbf{Z}_s(X;\Bbbk) \right\}\]
where
\[\mathbf{Z}_s(X;\Bbbk) =\ker\left(\bigotimes_{i=1}^s H^\ast(X;\Bbbk)\stackrel{\cup}{\to} H^\ast(X;\Bbbk)\right).\]
We have $\zcl_s(X) \le \TC_s(X)$, see \cite{BGRT}.

In some cases, a better lower bound can be obtained through the notion of $\TC_s$-weight. Recall (see \cite[\S 2]{FG}) that if $p:E\to B$ is a fibration and $u\in \widetilde{H}^*(B;\Bbbk)$ is a nontrivial class, the weight of $u$ associated to $p$, $\wgt_p(u)$, is the largest integer $k$ such that $f^*(u)=0$ for any map $f:Y\to B$ satisfying $\secat(f^*(p))<k$. If $u\neq 0$, then $\wgt_p(u)>0$ if and only if $p^*(u)=0$ and $\secat(p)\geq \wgt_p(u)$. Moreover, if $u_1,\dots,u_l\in \widetilde{H}^*(B;\Bbbk)$ satisfy $u_1\cup\cdots\cup u_l\neq 0$ then 
\[\wgt_p(u_1\cup\cdots\cup u_l)\geq \wgt_p(u_1)+\cdots+\wgt_p(u_l).\]
For a space $X$, the $\TC_s$-weight, denoted by $\wgt_s$, is the weight associated to the fibration $e_s:PX\to X^s$.  Taking coefficients in $\Bbbk$, the morphism $e_s^*$ can be identified with the $s$-fold cup-product and we can define the ($\Bbbk$-coefficients) \emph{weighted $s$-zero divisor cup length}, $\zclw(X)=\zclw(X;\Bbbk)$, to be the maximum of the set
\[\left\{ \sum_{i=1}^\ell\wgt_s(u_i)  \ |\ u_1\dots u_\ell {\neq 0, u_i} \in \mathbf{Z}_s(X; \Bbbk)\ \right\}.\]
We have $\TC_s(X)\geq \zclw(X)$. We also note that if $f:Y\to X$ is a map and $u\in H^*(X^s;\Bbbk)$ satisfies $(f^s)^*(u)\neq 0$ then $\wgt_s(f^*(u))\geq \wgt_s(u)$.\\

\subsection{Manifolds with \texorpdfstring{$\pi_1(M)=\Z_2$}{pi1(M)=Z/2} and \texorpdfstring{$\cat(M)=\dim M$}{cat(M)=dim(M)}}

The $\Z_2$-coefficient $s$-zero-divisor cuplength of the real projective space $P^n$ has been studied extensively, see \cite{CGGGL} and \cite{Davis}. We will see in Proposition \ref{prop:generalmanifold} below how to use these results to obtain information on $\TC_s(M)$ when  $\pi_1(M)=\Z_2$ and $\cat(M)=\dim M$. We first recall some results from \cite{Davis}.

For an integer $n>0$ with binary expansion $\cdots d_2d_1d_0$, i.e., $n=\sum_{i\geq0}d_i2^i$, with digits $d_i\in\{0,1\}$, let
\begin{itemize}
\item $\nu(n)$ denote the exponent in the maximal 2-power dividing $n$, i.e., $\nu(n)$ is the minimal $i$ with $d_i=1$;
\item $S(n):=\{i>0: \cdots d_{i+1}d_id_{i-1}\cdots=\cdots011\cdots\}$, the set of binary positions $i$ starting (from left to right) a block of consecutive 1's of length at least 2;
\item $Z_i(n):=\sum_{j=0}^{i}(1-d_j)2^j$, the complement of the binary expansion of $n$ mod $2^{i+1}$.
\end{itemize}
Building on \cite{CGGGL}, Davis \cite{Davis} proves that, for $s \ge 3$, the $\Z_2$-coefficient $s$-zero-divisors cuplength of the $n$-dimensional real projective space $P^n$ is given by
$$
\zcl_s(P^n)=sn-m_{n,s}
$$
where $m_{n,s}=\max\{2^{\nu(n+1)}-1,2^{i+1}-1-sZ_i(n):i\in S(n)\}$. In particular, for even $n$ (so that $P^n$ is non-orientable), $P^n$ has maximal possible $\TC_s(P^n)$, that is, $\TC_s(P^n)=sn$, whenever
\begin{equation}\label{sgrande}
s\geq\max\left\{3,\left\lceil\frac{2^{i+1}-1}{Z_i(n)}\right\rceil:i\in S(n)\right\}.
\end{equation}

We specialize two cases of the condition (\ref{sgrande}):
\begin{example}\label{afilada1}
 \begin{itemize}
\item[(a)]  
When $n$ is even and its binary expansion has no blocks of two or more consecutive 1's, we have $S(n)=\emptyset$ and the inequality \eqref{sgrande} reduces to $s\geq3$. Note that this condition for the maximality of $\TC_s(P^n)$ is sharp, since $\TC_2(P^n)<2n$ (\cite{FTY}, \cite[Theorem 1]{MR2649230}).
\item[(b)] For $n=2^{r+1}-2$, we have $S(n)=\{r\}$, $m_{n,s}=2^{r+1}-1-s$ and (\ref{sgrande}) becomes $s\geq 2^{r+1}-1$. Note that, when $s=n=2^{r+1}-2$, we have $\TC_n(P^n)\in \{sn,sn-1\}$. We will see in Section \ref{no-orient} that $\TC_n(P^n)=sn-1$ so that the condition $s\geq 2^{r+1}-1$ for the maximality of $\TC_s(P^n)$ is sharp again.
\end{itemize}
\end{example}

Thanks to the following result, Davis' computations of $\zcl_s(P^n)$ have impact on more general manifolds.

\begin{proposition} \label{prop:generalmanifold}
Let $M$ be an $n$-dimensional connected closed manifold with $\cat(M)=n$ and $\pi_1(M)=\Z_2$. Then, for any $s\geq 2$, $\TC_s(M)\geq \zcl_s(P^n; \Z_2)$.
\end{proposition}
\begin{proof} Let $\gamma:M\to P^\infty=B\Z_2$ be a classifying map and let $x\in H_1(B\Z_2;\Z_2)=\Z_2$ be the generator. For dimensional reasons, $\gamma$ factors as $M\stackrel{c}{\to} P^n\hookrightarrow  P^{\infty}$. Let $x_M=\gamma^*(x)=c^*(x)\in H^1(M;\Z_2)$. 
The hypothesis that $\cat(M)=n$ implies that $x_M\neq 0$, see \cite{Berstein}. 
Consequently $c^*:H^*( P^n;\Z_2)\to H^*(M;\Z_2)$ as well as $(c^s)^*:H^*(( P^n)^s;\Z_2)\to H^*(M^s;\Z_2)$ are monomorphisms. As the image by $(c^s)^*$ of a $s$-zero-divisor of $ P^n$ over $\Z_2$ gives rise to a $s$-zero-divisor of $M$ over $\Z_2$, we can conclude that $\zcl_s(M;\Z_2)\geq \zcl_s( P ^n;\Z_2)$ and the result follows.
\end{proof}

From Example \ref{afilada1} and the discussion above, we directly obtain: 
\begin{corollary}\label{cor:lowerboundZ2}
 Let $M$ be an $n$-dimensional connected closed manifold with $\cat(M)=n$ and $\pi_1(M)=\Z_2$. If $n$ is even and $s$ satisfies the inequality (\ref{sgrande}), then $\TC_s(M)=sn$. In particular:
 \begin{itemize}
 \item[(a)]  If $n$ is even and its binary expansion of $n$ contains no consecutive digits equal to $1$, then $\TC_s(M)=sn$ for any $s\geq 3$.
 \item[(b)] If $n=2^{r+1}-2$, then $\TC_s(M)=sn$ for any $s\geq 2^{r+1}-1$.
 \end{itemize}
\end{corollary}  

By using Davis' computations in combination with Proposition \ref{cyclic}, we can also state:
\begin{corollary}\label{Exorientable}
Let $M$ be an orientable connected closed manifold satisfying 
the conditions 
$n=\dim(M)=\cat(M)$ and $\pi_1(M)=\Z_2$.  If $n\equiv1\bmod4$, then $\TC_s(M)=sn-1$ for $s$ satisfying the inequality (\ref{sgrande}).
\end{corollary}
\begin{proof}
In this case $m_{n,s}=1$ so that $\TC_s(M)\geq sn-1$ for $s$ satisfying the inequality (\ref{sgrande}). The other direction follows from Proposition \ref{cyclic}.
\end{proof}
\begin{remark}
Observe that 
the orientability hypothesis, together with the condition $\dim(M)=\cat(M)$, implies that $n$ is odd. Indeed, if $n$ were even, the image of the fundamental class of $M$ by $\gamma_*:H_n(M)\to H_n(B\Z_2)=0$ would vanish. But this fact would force, by Poincar\'e duality, the $n$th power of the Berstein-Schwarz class of $M$ to vanish, contradicting the equality $\dim(M)=\cat(M)$.   
\end{remark}

\subsection{Odd dimensional manifolds with \texorpdfstring{$\pi_1(M)=\Z_p$}{pi1(M)=Z/p} and \texorpdfstring{$\cat(M)=\dim M$}{cat(M)=dim(M)}}
Throughout this section we consider a prime $p\geq 3$. Recall that the classifying space $B\Z_p$ can be identified with the infinite dimensional lens space $L_p^{\infty}$.
In order to have an analogue of Proposition \ref{prop:generalmanifold}, we first note the following result:
\begin{lemma}\label{lemmaClassyingmapZm}
    Let $M$ be an orientable connected closed $(2n+1)$-manifold   
    satisfying the conditions 
    $\cat(M)=2n+1$ and $\pi_1(M)=\Z_p$.  If $\gamma\colon M\to B\Z_p$ is a classifying map, then $\gamma_\ast([M])\in H_{2n+1}(\Z_p;\Z_p)$ is non-zero.
\end{lemma}
\begin{proof}
    Considering the Berstein--Schwarz class $\mathfrak b_M\in H^1(M;I(\pi))$, we have $\cat(M)=\dim(M)=2n+1$ if and only if $\mathfrak b_M^{2n+1}\neq 0$. By Poincar\'e duality, the second statement is equivalent to $\mathfrak b_M^{2n+1}([M])\neq 0$.  Taking cap products we obtain
    \[[M]\cap \mathfrak b_M^{2n+1} = [M]\cap \gamma^\ast(\mathfrak b^{2n+1}_{\Z_p}) \neq 0.\]
    Since $\gamma$ induces an isomorphism at the level of fundamental groups, naturality of the cap-products yields
   \[ \gamma_\ast([M])\cap\mathfrak b^{2n+1}_{\Z_p}\neq 0.\]
    Hence, $\gamma_\ast([M])\neq0$ in $H_{2n+1}(\Z_p;\Z)=\Z_p$. Since  $H_{2n+1}(\Z_p;\Z)=H_{2n+1}(\Z_p;\Z_p)=\Z_p$ we can conclude that $\gamma_\ast([M])\neq 0$ in $H_{2n+1}(\Z_p;\Z_p)$.
\end{proof}

\begin{theorem}\label{zclw_M_Lens}
    Let $M$ be a closed orientable $(2n+1)$-manifold with $\cat(M)=2n+1$ and $\pi_1(M)=\Z_p$ where $p\geq 3$ is a prime.  Then, $\TC_s(M)\geq\zclw(M; \Z_p)\geq\zclw(L_p^{2n+1};\Z_p)$.
\end{theorem}
\begin{proof} Let $\gamma\colon M\to B\Z_p$ be a classifying map.
We have a commutative triangle
    \begin{equation}\label{eqn.commTri_Lens}\begin{tikzcd}
        M \arrow[r,"\phi"] \arrow[dr,"\gamma"'] & L^{2n+1} \arrow[d, tail] \\ & L^\infty_p
    \end{tikzcd} \end{equation}
    where the inclusion is simply the $(2n+1)$-skeleton of the infinite dimensional lens space $L^\infty_p\simeq B\Z_p$. Since $\cat(M)=2n+1$, we know by the previous lemma that $\gamma_\ast([M])\neq 0$ in $H_{2n+1}(\Z_p;\Z_p)$. Consequently $\phi_\ast([M])\neq 0$ in $H_{2n+1}(L^{2n+1};\Z_p)$.
    Recall that \[H^\ast(L^{2n+1}_p;\Z_p)=\Z_p[x,y]/(y^{n+1},x^2)\]
    where $|x|=1$, $|y|=2$. In particular, $H^{2n+1}(L^{2n+1}_p;\Z_p)=\Z_p xy^n$.
We first check that $\phi^*:H^*(L_p^{2n+1};\Z_p)\to H^*(M;\Z_p)$ is a monomorphism. 
    It suffices to show that $\phi^\ast(x)$ and $\phi^\ast(y),\dots,\phi^\ast(y^n)$ are non-trivial in $H^\ast(M;\Z_p)$. Consider the following cap-product diagram:
 \begin{equation*}
\xymatrixcolsep{5pc}
\xymatrix{
	H_{2n+1}(M;\Z) \otimes H^{2n+1}(M;\Z_p) \ar@<-8ex>[d]^{\phi_*} \ar[r]^-{\cap}_-{\cong} & 
	H_0(M;\Z_p)=\Z_p \ar[d]_{=} \\
	H_{2n+1}(L_p^{2n+1};\Z_p) \otimes H^{2n+1}(L_p^{2n+1};\Z_p ) \ar@<-8ex>[u]_{\phi^*} 
	\ar[r]^-{\cap}_-{\cong} & 
H_0(L_p^{2n+1};\Z_p)=\Z_p	
}
\end{equation*}
Both horizontal morphisms are isomorphisms by Poincar\'e duality. As $\phi_\ast([M])\neq 0$ in $H_{2n+1}(L^{2n+1};\Z_p)$, $\phi_*([M]$ is not divisible by $p$ in $H_{2n+1}(L^{2n+1};\Z)=\Z$. Consequently, $\phi_*([M])\cap xy^n\neq 0$ and using the diagram we have that $[M]\cap\phi^\ast(xy^n)\neq 0$. We thus have $\phi^*(xy^n)\neq 0$ in $H^*(M;\Z_p)$ and hence $\phi^\ast(x)$ and $\phi^\ast(y^i)$ for $i=1,\dots,n$ are non-zero. Therefore $\phi^*$ is a monomorphism. By K\"unneth formula, we obtain that $(\phi^s)^*$ is also a monomorphism. Let $u\neq 0\in \mathbf{Z}_s(L_p^{2n+1}; \Z_p)$. Then $(\phi^s)^*(u)\in \mathbf{Z}_s(M;\Z_p)$ and  $(\phi^s)^*(u)\neq 0$. Consequently $\wgt_s((\phi^s)^*(u))\geq \wgt_s(u)$ and the results follows by definition of $\zclw$.
\end{proof}

There are extensive computations of $\zclw(L^{n}_p;\Z_p)$, see \cite[\S 5]{FG} and \cite[Section~5]{Daundkar}. By using the previous theorem, we can use the information coming from these computations for a larger class of manifolds.

\begin{corollary}\label{cor_LensSpacesGen}
    Let $s\geq 2$ and let $M$ be a closed orientable $(2n+1)$-manifold with $\cat(M)=2n+1$ and $\pi_1(M)=\Z_p$ where $p\geq 3$ is prime.  Then,
    \[\TC_s(M)\geq 
    \begin{cases}
        s\cdot(\ell+\ell'+1)-1 &\text{if $s$ is even};\\
        (s-1)\cdot(\ell+\ell')+s +2n-1 &\text{if $s$ is odd}
    \end{cases}\]
    where $0\leq \ell,\ell'\leq n$ are any integers such that $m$ does not divide $\binom{\ell+\ell'}{\ell}^{\lfloor s/2\rfloor}$.
\end{corollary}
\begin{proof}
    Here we apply the computation \cite[Theorem~5.2]{Daundkar} and Theorem~\ref{zclw_M_Lens}.
\end{proof}

In many situations, for example if $s$ is much larger than the dimension of $M$, we obtain an exact computation.

\begin{corollary} \label{exactvalueZm}
    Let $s\geq 2$ and let $M$ be a closed orientable $(2n+1)$-manifold with $\cat(M)=2n+1$ and $\pi_1(M)=\Z_p$ where $p\geq 3$ is a prime. If $s$ does not divide $\binom{2n}{n}^{\lfloor s/2\rfloor}$, then \[\TC_s(M)=s(2n+1)-1.\]
\end{corollary}
\begin{proof}
    The lower bound follows from Corollary~\ref{cor_LensSpacesGen} (see also \cite[Theorem~5.3]{Daundkar}) and the upper bound is Proposition~\ref{cyclic}.
\end{proof}

\section{Some calculations in the non-orientable case}\label{no-orient}
We now address the (non-)maximality of $\TC_s(M)$ for non-orientable manifolds having $\pi_1(M)=\Z_2$. The case $s=2$ is well understood (\cite[Theorem 1]{MR2649230}), so we assume $s\geq3$ from now on. For such cases the non-maximality of $\TC_s(M)$ demands further restrictions on $s$. The aim of this section is to establish the following result.

\begin{theorem}\label{calculation1}
Let $M$ be a non-orientable $n$-dimensional manifold with $\pi_1(M)=\mathbb{Z}_2$ and $n=2^{r+1}-2$. Then, for any even $s$ no greater than $2^{r+1}-2$, we have $\TC_s(M)<sn$.
\end{theorem}

By Corollary \ref{cor:lowerboundZ2}(b), we know that, for $n=2^{r+1}-2$, $\TC_s(M)=sn$ for $s\geq 2^{r+1}-1$, so that in this case the upper limiting restriction on $s$ in Theorem \ref{calculation1} is in fact sharp and we have:

\begin{corollary}\label{p6}
If the manifold $M$ in Theorem \ref{calculation1} has $\cat{M}=n$, then $\TC_s(M)=sn-1$ for $s=2^{r+1}-2$ and $\TC_s(M)=sn$ for $s\geq 2^{r+1}-1$.
\end{corollary}
\begin{proof}
The equality $\TC_s(M)=sn-1$ for $s=2^{r+1}-2$ follows from $m_{2^{r+1}-2,2^{r+1}-2}=1$ (see Example \ref{afilada1}(b)), Proposition \ref{prop:generalmanifold} and Theorem \ref{calculation1}.
\end{proof}

Corollary \ref{p6} should be compared to the fact that $\TC_s(P^{2^r})$ is maximal for $s\geq3$, but $\TC_2(P^{2^r})=\Imm(P^{2^r})=2^{r+1}-1$ (\cite{FTY}). Worth noting is the fact that the case $M=P^6$ in Corollary \ref{p6} (with $r=2$) upgrades the observation in \cite[(7.4)]{CGGGL} that $\delta_6(6)\leq1$ to an equality, giving evidence for what would be regular behavior of the higher topological complexity of projective spaces $P^m$ with $m=2^a+2^{a+1}$.

Suitable analogues of Theorem \ref{calculation1} should hold for more general values of $n$, but the complexity of calculations seems to be a major obstacle towards obtaining corresponding proofs.

We now start working towards the proof of Theorem \ref{calculation1}. From now on $\pi:=\Z_2$ and $s=2\sigma$ with $1\leq\sigma\leq2^r-1=n/2$. Set $\widehat{\Z}:=\widetilde{\Z}\otimes(\Z\otimes\widetilde{\Z})^{\sigma-1}$, the $\Z[\pi^{s-1}]$-module of Remark \ref{elsistemadecorficientes}. By Proposition \ref{suff-cond}, it suffices to establish the triviality of
\begin{equation}\label{laobstruccion}
^s\chi_*(\mathfrak{m}^{\times s})\in H_{sn}(\pi^{s-1};\widehat{\mathbb{Z}})
\end{equation}
where
\begin{equation}\label{chi}
^s\chi_*\colon H_*(\pi^s;\widetilde{\mathbb{Z}}^{\otimes s})\to H_*(\pi^{s-1};\widehat{\mathbb{Z}}).
\end{equation}

As in \S\ref{sec:orient}, 
our starting point is the free $\Z[\Z_q]$-resolution (\ref{resolutionZq})  of $\Z$ with $q=2$. Recall that $[k]$ denotes the generator of degree $k$. In addition to the chain complex $C_{\bullet}(\pi)$ of (\ref{zcc}), we will also need the complex $\widetilde{C}_{\bullet}(\pi)$,
\begin{equation}\label{lactilde}\xymatrix{
\ar[r]^{-2}&\Z[2k]  \ar[r]^{0}& \Z[2k-1] \ar[r]^{-2} &\cdots \ar[r]^{0}&\Z[1]  \ar[r]^{-2}& \Z[0]  
},\end{equation}
obtained by tensoring (\ref{resolutionZq}) with $\widetilde{\Z}$ over $\pi$. Abusing notation, we continue using $[k]$ for the generators of both $C_{\bullet}(\pi)$ and $\widetilde{C}_{\bullet}(\pi)$.

The homology groups in (\ref{chi}) can be computed from the complexes $\widetilde{C}_{\bullet}(\pi)^{\otimes s}$ and
\begin{equation}\label{complexwhereobstructionlies}
\mathcal{D}_{\bullet}:=\widetilde{C}_{\bullet}(\pi)\otimes({C}_{\bullet}(\pi)\otimes \widetilde{C}_{\bullet}(\pi))^{\sigma-1}.
\end{equation}
In both cases, we will use the shorthand $[i_1,\ldots,i_\ell]$ for a tensor product $[i_1]\otimes\cdots\otimes[i_\ell]$. The K\"unneth formula and the fact that the homology of $\widetilde{C}_{\bullet}(\pi)$ is 2-torsion (in all degrees) gives:

\begin{lemma}\label{detorsion}
The element $^s\chi_*(\mathfrak{m}^{\times s})$ in (\ref{laobstruccion}) is torsion. Indeed, both groups in (\ref{chi}) are torsion.
\end{lemma}

Let $\mathcal{H}_{\bullet}$ denote the quotient of $\mathcal{D}_{\bullet}$ resulting from killing all boundaries, and consider the obvious monomorphism $\iota:H_{\bullet}(\pi^{s-1};\widehat{\mathbb{Z}})\hookrightarrow\mathcal{H}_{\bullet}$. The triviality of the element in (\ref{laobstruccion}) follows from Lemma \ref{detorsion} and the following key result, whose proof is addressed in the rest of the section through a direct analysis of (\ref{laobstruccion}) and (\ref{chi}).
\begin{proposition}\label{calculofinal}
The class $\iota(^s\chi_*(\mathfrak{m}^{\times s}))$ is an element of the torsion-free (graded) subgroup of $\mathcal{H}_{\bullet}$.
\end{proposition}

In computing the homology groups in (\ref{chi}) using the complexes $\widetilde{C}_{\bullet}(\pi)^{\otimes s}$ and $\mathcal{D}_{\bullet}$, we will use $\widetilde{\pi}$ for a factor where $\widetilde{C}_{\bullet}(\pi)$ is meant to be taken, reserving the notation $\pi$ for factors where $C_{\bullet}(\pi)$ is meant to be taken. For instance, the diagonal morphism $\Delta:\pi\to\pi\times\pi$ and the group-multiplication morphism $\mu:\pi\times\pi\to\pi$ extend to morphisms $\Delta:\widetilde{\pi}\to\widetilde{\pi}\times\pi$, $\Delta:\widetilde{\pi}\to\pi\times\widetilde{\pi}$ and $\mu:\widetilde{\pi}\times\widetilde{\pi}\to\widetilde{\pi}$ that are compatible with the implied module structures. In these terms, since 
$\pi=\Z_2=\langle v \mid v^2=1\rangle$ in the present case, 
the inversion morphism plays no role and the map $^s\chi$ factors as
\begin{equation}
\begin{aligned}
\widetilde{\pi}^s\stackrel{1\times (\Delta\times\Delta)^{\sigma-1}\times1}{-\!\!\!-\!\!\!-\!\!\!-\!\!\!-\!\!\!-\!\!\!-\!\!\!-\!\!\!-\!\!\!-\!\!\!\longrightarrow} {}&
\widetilde{\pi}\times(\widetilde{\pi}\times\pi\times\pi\times\widetilde{\pi})^{\sigma-1}\times\widetilde{\pi} \label{lacomposicion}\\  =&(\widetilde{\pi}\times\widetilde{\pi})\times(\pi\times\pi\times\widetilde{\pi}\times\widetilde{\pi})^{\sigma-1}\stackrel{\mu\times(\mu\times\mu)^{\sigma-1}}{-\!\!\!-\!\!\!-\!\!\!-\!\!\!-\!\!\!-\!\!\!-\!\!\!\longrightarrow}\widetilde{\pi}\times(\pi\times\widetilde{\pi})^{\sigma-1}.
\end{aligned}    
\end{equation}

Recalling  that $B_{i,j}$ denotes the binomial coefficient $\binom{\,i+j\,}{i}$, the {formul\ae}  of Lemma \ref{deltaymuenresoluciones} $(b)$ written with the shorthand in use in this section read
\begin{equation} \label{cerito}
\begin{aligned}
{}[2i,2j]&\mapsto B_{i,j}\,[2(i+j)]; \\ 
{}[2i,2j+1]&\mapsto B_{i,j}\,[2(i+j)+1]; \\ 
{}[2i+1,2j]&\mapsto B_{i,j}\,[2(i+j)+1]; \\ 
{}[2i+1,2j+1]&\mapsto 0. 
\end{aligned}
\end{equation}
We note also that, since $\pi=\Z_2$, the formula of Lemma \ref{deltaymuenresoluciones}(a) giving $\Delta$ on generators at the level of resolutions can be written
\begin{equation}\label{deltaresolucionesZ2}
[k]\to\sum_{p+q=k}[p]\otimes v^{\mathrm{odd}(p)}\hspace{-.6mm}\cdot [q],
\end{equation}
where $v$ generates $\pi$ and ${\mathrm{odd}(p)}=1$ if $p$ is odd and $0$ otherwise.

The class $\mathfrak{m}\in H_*(\pi;\widetilde{\mathbb{Z}})$ is either trivial or, else, represented by the cycle $[n]$ in (\ref{lactilde}) ---recall $n$ is even. For the purposes of proving Theorem \ref{calculation1}, we may safely assume the latter possibility. Then, $\mathfrak{m}^{\times s}$ is represented in $\widetilde{C}_{\bullet}(\pi)^{\otimes s}$ by the corresponding tensor product $[n,n,\ldots,n]$. We chase the latter element under the first map of the composite (\ref{lacomposicion}) to get, in view of (\ref{deltaresolucionesZ2}),
\begin{align*}
[\widetilde{n},\widetilde{n},\ldots,\widetilde{n}]\mapsto
[\widetilde{n}]\otimes\left(\hspace{-1mm}\left(\,\sum_{p+q=n}[\widetilde{p}]\otimes v^{\text{odd}(p)}\cdot[q]\right)\otimes\left(\,\sum_{p+q=n}[p]\otimes v^{\text{odd}(p)}\cdot[\widetilde{q}]\right)\hspace{-1mm}\right)^{\hspace{-1mm}\otimes \sigma-1}\hspace{-2mm}\otimes[\widetilde{n}].
\end{align*}
Note that in the latter expression we are extending in the obvious way the convention above regarding the use of $\pi$ and $\widetilde{\pi}$. Then, after tensoring with the needed coefficients (thus dropping the $\sim$ indicators), this becomes
\begin{align*}
[{n},{n},\ldots,{n}]\mapsto &
[{n}]\otimes\left(\hspace{-1mm}\left(\,\sum_{p+q=n}[p,q]\right)\otimes\left(\,\sum_{p+q=n}(-1)^p[p,q]\right)\hspace{-1mm}\right)^{\hspace{-1mm}\otimes \sigma-1}\hspace{-2mm}\otimes[{n}]\\
=&\sum_{\arraba{p_i+q_i=n}{1\leq i\leq s-2}}(-1)^{{}^{\sum_{j=1}^{\sigma-1}p_{2j}}}[n,p_1,q_1,p_2,q_2,\cdots,p_{s-2},q_{s-2},n].
\end{align*}
Since $n$ is even, the parity of each $p_i$ agrees with the one of the corresponding $q_i$, so the last expression can be rewritten as
\begin{align*}
\sum(-1)^{{}^{\sum_{j=1}^{\sigma-1}\delta_{2j}}}[n,2p_1+\delta_1,2q_1+\delta_1,2p_2+\delta_2,2q_2+\delta_2,\cdots\hspace{-.3mm},2p_{s-2}+\delta_{s-2},2q_{s-2}+\delta_{s-2},n],
\end{align*}
where the sum now runs over $1\leq i\leq s-2$, $\delta_i\in\{0,1\}$ and $p_i+q_i=n/2-\delta_i$. Using the {formulae} (\ref{cerito}), we finally obtain the image of $[n,\ldots,n]$ under the entire composition in (\ref{lacomposicion}). This image may be expressed as
\begin{equation} \label{laobstrucciondesarrollada}
[n,\ldots,n] \mapsto
\sum (-1)^{{}^{\sum_{j=1}^{\sigma-1}\delta_{2j}}} B_{\frac{n}{2},p_1}B_{q_1,p_2}\cdots B_{q_{s-3},p_{s-2}}B_{q_{s-2},\frac{n}2} \cdot \mathbf{G},   
\end{equation}
where
\[
\mathbf{G}=[n+2p_1+\delta_1,2(q_1+p_2)+\delta_1+\delta_2,\ldots,2(q_{s-3}+p_{s-2})+\delta_{s-3}+\delta_{s-2},n+2q_{s-2}+\delta_{s-2}]
\]
and the sum runs over the same indices as above, except now that no two consecutive $\delta_j$ and $\delta_{j+1}$ can simultaneously equal 1, in view of (\ref{cerito}). In what follows we set $m={n}/{2}$.

\medskip
Having described a cycle representing the obstruction in (\ref{laobstruccion}), we next spell out the complex (\ref{complexwhereobstructionlies}) where it lies. Degreewise, $\mathcal{D}_{\bullet}$ is $\mathbb{Z}$-free with basis given by elements $[u_1,v_1,\ldots,u_{\sigma-1},v_{\sigma-1},u_\sigma]$ for non-negative integers $u_i$ and $v_i$, and with differential
\begin{align}
\partial[u_1,&v_1,\ldots,u_{\sigma-1},v_{\sigma-1},u_\sigma]=\nonumber\\
=&-2\hspace{.8mm}\text{odd}\hspace{.4mm}(u_1)
[u_1-1,v_1,u_2,v_2,\ldots,u_{\sigma-1},v_{\sigma-1},u_\sigma] \nonumber\\
&\quad +(-1)^{u_1}\hspace{.3mm}2 \hspace{.8mm}\text{even}\hspace{.4mm}(v_1)
[u_1,v_1-1,u_2,v_2,\ldots,u_{\sigma-1},v_{\sigma-1},u_\sigma] \nonumber\\
&\quad -(-1)^{u_1+v_1}\hspace{.3mm}2 \hspace{.8mm}\text{odd}\hspace{.4mm}(u_2)
[u_1,v_1,u_2-1,v_2,\ldots,u_{\sigma-1},v_{\sigma-1},u_\sigma] \nonumber \\
&\quad +(-1)^{u_1+v_1+u_2}\hspace{.3mm}2 \hspace{.8mm}\text{even}\hspace{.4mm}(v_2)
[u_1,v_1,u_2,v_2-1,\ldots,u_{\sigma-1},v_{\sigma-1},u_\sigma] \label{ladiferencial} \\
&\quad \pm \cdots \nonumber\\
&\quad -(-1)^{u_1+v_1+\cdots+u_{\sigma-2}+v_{\sigma-2}}\hspace{.3mm}2 \hspace{.8mm}\text{odd}\hspace{.4mm}(u_{\sigma-1})
[u_1,v_1,u_2,v_2,\ldots,u_{\sigma-1}-1,v_{\sigma-1},u_\sigma] \nonumber\\
&\quad +(-1)^{u_1+v_1+\cdots+u_{\sigma-2}+v_{\sigma-2}+u_{\sigma-1}}\hspace{.3mm}2 \hspace{.8mm}\text{even}\hspace{.4mm}(v_{\sigma-1})
[u_1,v_1,u_2,v_2,\ldots,u_{\sigma-1},v_{\sigma-1}-1,u_\sigma] \nonumber\\
&\quad -(-1)^{u_1+v_1+\cdots+u_{\sigma-1}+v_{\sigma-1}}\hspace{.3mm}2 \hspace{.8mm}\text{odd}\hspace{.4mm}(u_{\sigma})
[u_1,v_1,u_2,v_2,\ldots,u_{\sigma-1},v_{\sigma-1},u_\sigma-1],\nonumber
\end{align}
where a basis element with a negative entry is meant to be interpreted as zero.

\medskip
All the information needed to prove Proposition \ref{calculofinal} is of course contained in (\ref{laobstrucciondesarrollada}) and (\ref{ladiferencial}). The following constructions are meant to organize a proof argument.

\begin{definition}
For a positive integer $k$, let $p(k)$ denote the set of binary positions where the binary expansion of $k$ has digit 1. For instance, $p(5=4+1)=\{0,2\}$ and $p(42=32+8+2)=\{1,3,5\}$.
\end{definition}

A standard well known fact is:
\begin{lemma}\label{wellknwn}
A binomial coefficient $B_{a,b}$ is even if and only if $p(a)\cap p(b)\neq\varnothing$.
\end{lemma}

The following result is the only place where the special assumptions in Theorem \ref{calculation1} (i.e., $s=2\sigma\leq n=2m=2^{r+1}-2$ with $r\geq1$) are needed. 
\begin{proposition}\label{coeficientesbinomiales}
Any coefficient
\begin{equation}\label{productodebinomiales}
B_{m,p_1}B_{q_1,p_2}\cdots B_{q_{s-3},p_{s-2}}B_{q_{s-2},m}
\end{equation}
in (\ref{laobstrucciondesarrollada}) with $p_1+q_1=m$ (i.e., with $\delta_1=0$) is even.
\end{proposition}
\begin{proof}
We use without further notice the fact coming from Lemma \ref{wellknwn} that any binomial coefficient $B_{m-j,i}$ is even whenever $0\leq j<i\leq m$. Recall $m=2^r-1$ and
\begin{equation}\label{cotaensigma}
s=2\sigma\leq 2^{r+1}-2
\end{equation}
with $r\geq1$. Assume for a contradiction that some coefficient (\ref{productodebinomiales}) is odd (i.e., that all of its binomial-coefficient factors are odd) and has $\delta_1=0$. Recall the forced conditions
\begin{enumerate}
\item \label{i1} $p_i+q_i=m-\delta_i$ with $p_i,q_i\geq0$ and $\delta_i\in\{0,1\}$;
\item \label{i3} $1\leq i<s-2$ and $\delta_i=1$ implies $\delta_{i+1}=0$,
\end{enumerate}
for $1\leq i\leq s-2$. Let $2\leq i_1<i_2<\cdots<i_k\leq s-2$ be all the indices $j$ (if any) with $\delta_j=1$. Note that
\begin{equation}\label{cotadek}
0\leq k\leq \sigma-1,
\end{equation}
in view of (\ref{i3}).

The coefficient $B_{m,p_1}$ is odd by hypothesis, so $p_1=0$ and $q_1=m$ ---the latter equality holds in view of (\ref{i1}) since $\delta_1=0$. Actually, the same argument can be used iteratively for $1\leq j<i_1$ (so $\delta_j=0$) with the binomial coefficients $B_{q_{j-1},p_j}$ (e.g. $q_0:=m$) to show that
$$
\mbox{$p_j=0$ \, and \, $q_j=m$.}
$$
Next, since $m=q_{i_1-1}$, $B_{q_{i_1-1},p_{i_1}}$ is odd and $\delta_{i_1}=1$, we get
$$
\mbox{$p_{i_1}=0$ \, and \, $q_{i_1}=m-1$,}
$$
and now the process repeats with a slight adjustment. For starters, $q_{i_1}=m-1$, $B_{q_{i_1},p_{i_1+1}}$ is odd and $\delta_{i_1+1}=0$ is forced by (\ref{i3}), so that $p_{i_1+1}\leq1$ and $q_{i_1+1}\geq m-1$. We then iterate the latter argument: For $i_1< j<i_2$, the assumption $\delta_j=0$ and the fact that $B_{q_{j-1},p_j}$ is odd with $q_{j-1}\geq m-1$ yield
$$
\mbox{$p_j\leq 1$ \, and \, $q_j\geq m-1$.}
$$
Of course, the last two inequalities now hold for all $1\leq j<i_2$. The next round of iterations start with the fact that $B_{q_{i_2-1},p_{i_2}}$ is odd with $q_{i_2-1}\geq m-1$ and $\delta_{i_2}=1$, to get
$$
\mbox{$p_{i_2}\leq 1$ \, and \, $q_{i_2}\geq m-2$,}
$$
and the process has a corresponding new obvious adjustment to yield
$$
\mbox{$p_j\leq 2$ \, and \, $q_j\geq m-2$,}
$$
 for $1\leq j<i_3$, whereas
$$
\mbox{$p_{i_3}\leq 2$ \, and \, $q_{i_3}\geq m-3$.}
$$
Just before the last adjustment we get
$$
\mbox{$p_j\leq k-1$ \, and \, $q_j\geq m-k+1$,}
$$
for $1\leq j<i_k$, whereas
$$
\mbox{$p_{i_k}\leq k-1$ \, and \, $q_{i_k}\geq m-k$.}
$$
However, after this point the conditions $p_j\leq k$ and $q_j\geq m-k$ are kept for all $j\leq s-2$. In particular, $q_{s-2}\geq m-k=2^r-1-k\geq1$, in view of (\ref{cotaensigma}) and (\ref{cotadek}). But then the final factor $B_{q_{s-2},m}$ of (\ref{productodebinomiales}) is even, a contradiction.
\end{proof}

\begin{proof}[Proof of Proposition \ref{calculofinal}]
The right hand-side of (\ref{ladiferencial}) yields the defining relations in $\mathcal{H}_\bullet$. Namely, for each tuple $(u_1,v_1,\ldots,u_{\sigma-1},v_{\sigma-1},u_\sigma)$ of non-negative integers there is a defining relation
\begin{equation}\label{comorelacion}
0=U_1+V_1+U_2+V_2+\cdots+U_{\sigma-1}+V_{\sigma-1}+U_\sigma
\end{equation}
where
\begin{align*}
U_i:=&(-1)^{p_i}\,2\,\text{odd}(u_i)\,[u_1,v_1,\ldots,u_{i-1},v_{i-1},u_i-1,v_i,u_{i+1},v_{i+1},\ldots,u_{\sigma-1},v_{\sigma-1},u_\sigma],\\
V_i:=&(-1)^{q_i}\,2\,\text{even}(v_i)\,[u_1,v_1,\ldots,u_{i-1},v_{i-1},u_i,v_i-1,u_{i+1},v_{i+1},\ldots,u_{\sigma-1},v_{\sigma-1},u_\sigma],\\ 
p_i:=&1+\sum_{1\leq j<i}(u_j+v_j) \text{ \ \ and \ \ }
q_i:=\sum_{1\leq j<i}(u_j+v_j)+u_i.
\end{align*}

The tuple $(u_1,v_1,\ldots,u_{\sigma-1},v_{\sigma-1},u_\sigma)$ and the basis element $[u_1,v_1,\ldots,u_{\sigma-1},v_{\sigma-1},u_\sigma]$ of $\mathcal{D}_\bullet$ are said to be \emph{even} (respectively, \emph{odd}) when $u_1$ is even (respectively, odd). In the odd case, (\ref{comorelacion}) gives a way to write the double of the class in $\mathcal{H}_\bullet$ represented by an even basis element as a linear combination of the doubles of classes represented by odd basis elements. On the other hand, in the even case $U_1=0$ and the right hand-side of (\ref{comorelacion}) is a linear combination of doubles of classes represented by even basis elements. A straightforward computation\footnote{The calculation is formally identical to the standard verification that $\partial^2=0$ for the boundary morphism $\partial$ in the singular complex of a given space. Details are left as an exercise for the reader.} shows that the latter linear combination vanishes directly in $\mathcal{D}_\bullet$ when (the double of) each even summand is replaced by the corresponding linear combination of odd basis elements. Thus, relations (\ref{comorelacion}) coming from even tuples are irrelevant. Since each even basis element appears in a single relation (\ref{comorelacion}) coming from an odd tuple, we see that the subgroup of $\mathcal{H}_\bullet$ spanned by the classes represented by odd basis elements is torsion free. The proof is complete in view of Proposition \ref{coeficientesbinomiales}.
\end{proof}

\section*{Acknowledgments} 
The research of L. Vandembroucq was partially financed by Portuguese
Funds through FCT (Fundaç\~ao para a Ci\^encia e a Tecnologia) within the Project UID/00013:
Centro de Matemática da Universidade do Minho (CMAT/UM).
Sam Hughes was supported by a Humboldt Research Fellowship at Universit\"at Bonn.

\end{document}